\documentclass{amsart}

\usepackage{amsfonts}
\usepackage{cases}
\usepackage{mathrsfs}
\usepackage{bbm}
\usepackage{amssymb}
\usepackage{txfonts}
\usepackage{amscd}
\usepackage{amsfonts,latexsym,amsmath,amsthm,amsxtra,mathdots}
\usepackage[bookmarks,colorlinks]{hyperref}
\usepackage[all,cmtip]{xy}
\RequirePackage{amsmath} \RequirePackage{amssymb}
\usepackage{color}
\usepackage{colordvi}
\usepackage{multicol}
\usepackage{tikz}
\usepackage{hyperref}
\usepackage{graphicx}
\usepackage{amsmath}
\usepackage{amsmath,amscd}
\usepackage[normalem]{ulem}
\usetikzlibrary{matrix,arrows}
\usepackage{textcomp}

\makeatletter
\newenvironment{named}[1]
  {\def\namedthmname{#1}%
   \refstepcounter{namedthm}%
   \namedthm\def\@currentlabel{#1}}
  {\endnamedthm}
\makeatother

     \newcommand{\BR}{{\mathbb {R}}}

     \newcommand{\BZ}{{\mathbb {Z}}}

    \newcommand{\Hom}{{\mathrm{Hom}}}

     \newcommand{\Ext}{{\mathrm{Ext}}}

\def\-{^{-1}}

\newcommand{\delete}[1]{}

    \theoremstyle{plain}

\newtheorem{thm}{Theorem}[section]
\newtheorem{defn}[thm]{Definition} 

\newtheorem{lem}[thm]{Lemma}
\newtheorem{prop}[thm]{Proposition}
\newtheorem{cor}[thm]{Corollary}

\newtheorem*{namedthm}{\namedthmname}
\newcounter{namedthm}

\newtheorem*{thm*}{Theorem}

\newtheorem*{rem*}{Remark}

\theoremstyle{definition}
\newtheorem{rem}[thm]{Remark}
\newtheorem{example}[thm]{Example}

    \numberwithin{equation}{section}

\def\Proof{\noindent{\bf Proof}\quad}
\def\qed{\hfill$\square$\smallskip}

\begin{document}

\title{On the homotopy  of closed manifolds and finite CW-complexes}


\author{Yang Su}
\address{HLM, Academy of Mathematics and Systems Science, Chinese Academy of Sciences, Beijing 100190, China \\
School of Mathematical Sciences, University of Chinese Academy of Sciences, Beijing 100049, China
}
\email{suyang@math.ac.cn}


\author{Xiaolei Wu}
\address{University of Bonn, Mathematical Institute, Endenicher Allee 60, 53115 Bonn, Germany}
\email{xwu@math.uni-bonn.de}


\date{September, 2019}


\begin{abstract}
We study the finite generation of homotopy groups of closed manifolds and finite CW-complexes by relating it to the cohomology of their fundamental groups. Our main theorems are as follows: when $X$ is a finite CW-complex of dimension $n$ and $\pi_1(X)$ is  virtually a Poincar\'e duality group of dimension $\geq n-1$, then $\pi_i(X)$ is not finitely generated for some $i$ unless $X$ is homotopy equivalent to the Eilenberg--MacLane space $K(\pi_1(X),1)$; when $M$ is an $n$-dimensional closed manifold and $\pi_1(M)$ is virtually a Poincar\'e duality group of dimension $\ge n-1$,  then for some $i\leq [n/2]$, $\pi_i(M)$ is not finitely generated, unless $M$ itself is an aspherical manifold.  These generalize  theorems of M. Damian from polycyclic groups to any virtually Poincar\'e duality groups. When $\pi_1(X)$  is not a virtually Poincar\'e duality group, we also obtained similar results. As a by-product we showed that if a group $G$ is of  type F and  $H^i(G,\mathbb{Z} G)$ is finitely generated for any $i$, then $G$ is a  Poincar\'e duality group. This recovers partially a theorem of Farrell.
\end{abstract}

\maketitle
\section*{Introduction}
The homotopy groups $\pi_{i}(X)$ are important algebraic topological invariants associated to a space $X$. The study of their properties is a major topic in topology. The first concern is whether these abelian groups are finitely generated. Let $X$ be a finite connected CW-complex of dimension $n$. A celebrated theorem of Serre \cite{Se53} says that when $X$ is simply connected, then all the homotopy groups of $X$ are finitely generated. So it is a natural question to consider the case when $X$ is not simply connected. When $\pi_1(X)$ is finite, one can pass to the universal cover of $X$ which is still a finite CW-complex, hence Serre's theorem applies. When $\pi_1(X)$ is not finite, M. Damian did some interesting  work on this problem \cite{Da05,Da09}.  One of his main results \cite[Theorem 1.2]{Da09} says that when $X$ is a finite CW-complex of dimension $n$ and $\pi_1(X)$ is a polycyclic group of Hirsch length $\geq n-1$, then $\pi_i(X)$ is not finitely generated for some $i \ge 2$ or $X$ is aspherical. Recall that a topological space is aspherical if its universal cover is contractible.

In this paper, we relate the (non)-finite generation of homotopy groups  of closed manifolds and finite complexes to the cohomology of their fundamental groups. The main results extend Damian's theorems \cite[Theorem 1.2, 1.3, 1.4]{Da09} to a much broader class of groups. For example, our theorems apply to the case when the fundamental group is virtually a Poincar\'e duality group (see Definition \ref{defn-pdg}), in particular it holds for the fundamental group of any closed aspherical manifold.

It is well-known that any finitely presented group  can be realized as the fundamental group of a closed manifold of dimension $\ge 4$. The following theorem considers the case when the fundamental group is of type F (see Definition \ref{defn-type}).

\begin{named}{Theorem A}\label{mthm-pdg-mfd}
Let $M^n$ be a closed $n$-dimensional manifold with
$\pi_1(M)=G$, suppose $G$ is of type F, then

\begin{enumerate}
\item if $G$ is not a Poincar\'e duality group, then $\pi_{i}(M)$ is not finitely generated  for some $i \ge 2$ . Furthermore, if $H^{i}(G,\mathbb ZG)$ is finitely generated for all $i \le [n/2]$, then $\pi_{i}(M)$ is not finitely generated for some $2 \le i \le [n/2]$ ;

\item if $G$ is a duality group of dimension $d$ such that $d \ge n-1$,
then  either $\pi_i(M)$ is not finitely generated for some $2 \le i\leq [n/2]$, or  $M$ itself is aspherical.
\end{enumerate}
\end{named}

\begin{rem}
Note that the conclusions of the theorem still hold when the assumption is virtually satisfied, i.~e.~ if some finite index subgroup of $\pi_1(M)$ satisfies the assumption. In particular, this applies to polycyclic groups since they are virtually Poincar\'e duality groups (\cite[Theorem 2]{Hi52} and \cite{AJ76}).
\end{rem}

\begin{rem}
If $G$ is a Poincar\'e duality group of dimension $d$ then there exists a closed manifold $M$ with $\pi_{1}(M)=G$ and $\pi_{i}(M)$ finitely generated for all $i \ge 2$ (see \ref{thm-existence} below). From \ref{mthm-pdg-mfd} (b) we see that either $M$ is aspherical with fundamental group $G$, or $\dim M \ge  d+2$. In the first case the Borel conjecture predicts that $M$ is topologically rigid. It would be interesting to have a structure theorem for the manifolds in the second case. There are several works in this direction, including the classical fibration theorem of Browder-Levine \cite{BL} saying that if $G = \mathbb Z$ then $M$ is always a fiber bundle over $S^{1}$ with simply-connected fiber, and the analysis of the toplogical rigidity of some classes of these manifolds by Kreck and L\"uck \cite{KL06}.
\end{rem}

\begin{rem}
Note that by a result of C. Stark \footnote{We learned C. Stark¡¯s results \cite[Theorem 4.1]{St96} and \cite[Corollary 2]{St95} (Compare Remark \ref{rem-quesI}) after this work was completed.} \cite[Theorem 4.1]{St96} (see also Proposition \ref{prop:fgh-finf}), if the homotopy groups of a finite CW complex $X$ are all finitely generated, then $\pi_1(X)$ is of type $F_{\infty}$.
\end{rem}

In general, when the fundamental group is not necessarily of type $F$, we have the following.

\begin{named}{Theorem B}\label{mthm-gel-mfd}
Let $M^n$ be a closed $n$-dimensional manifold with
$\pi_1(M)=G$.

\begin{enumerate}

\item Suppose that $H^{i}(G,\BZ G) = 0$ for all $i\leq d$. If $n\leq d$, then for some $i$, $\pi_i(M)$ is not finitely generated. If $n =d+1$ or $d+2$, then either $\pi_i(M)$ is not finitely generated for some $i$, or $M$ is aspherical and  $G$ is a Poincar\'e duality group of dimension $n$.
\item if $H^{i}(G, \mathbb ZG)$ is not finitely generated for some $i$, then $\pi_j(M)$ is not finitely generated for some $j$.
\end{enumerate}
\end{named}

\begin{example}\label{ex}
The Thompson group $F$ is a finitely presented group of infinite cohomological dimension, but with the property that $H^{i}(F,\BZ F) = 0$ for all $i$ \cite[Theorem 7.2]{BG84}. Now part (a) of   \ref{mthm-gel-mfd} implies that any closed manifold with fundamental group $F$ can not have finitely generated homotopy group in each dimension.
\end{example}

Similar to \ref{mthm-pdg-mfd}, in the finite CW-complex case,  we have the following(compare \cite[Theorem 1.2]{Da09}).


\begin{named}{Theorem C}\label{mthm-pdg-cw}
Suppose that $X$ is a finite  CW-complex of dimension $n$ and that $\pi_1(X)$ is a virtually Poincar\'e duality group of dimension $d$. Then
\begin{enumerate}
    \item \label{casea} If $d > n$, then $\pi_i(X)$ is not finitely generated for some $i \ge 2$.
    \item \label{caseb} If $d= n$ or $n-1$, then $\pi_i(X)$ is not finitely generated for some $i \ge 2$ unless $X$ is a $K(\pi_1(X),1)$ space.
    \end{enumerate}
In particular, when $\pi_1(X)$ has torsion and $d \ge n-1$, the conclusion of (a) holds.
\end{named}

In general, if $\pi_1(X)$ is not a virtually Poincar\'e duality group, we have the following theorem for finite CW-complexes.

\begin{named}{Theorem D}\label{mthm-gel-cw}
Let $G$ be a group and $n$ be the smallest integer such that $H^{n+1}(G,\BZ G)$ is not finitely generated and $X$ is a  finite CW-complex  of dimension $\leq n$ with fundamental group $G$. Then for some $i\geq 2$, $\pi_i(X)$ is not  finitely generated.

\end{named}

Both \ref{mthm-pdg-cw} and \ref{mthm-gel-cw} are proved by transforming the problem into the manifold case. In general, it is an interesting problem to realize a given group as the fundamental group of manifolds satisfying certain topological conditions, such as knot complements \cite{Ke65} or homology spheres \cite{Ke69}. From this point of view  \ref{mthm-gel-mfd}  leads to the following question

\begin{named}{Question I}\label{ques}
Given a finitely presented group $G$ of finite cohomological dimension, such that $H^{i}(G, \mathbb ZG)$ is finitely generated for all $i$. Then does there exist a closed manifold $M$ with $\pi_{1}(M)=G$ and $\pi_{i}(M)$ finitely generated for all $i \ge 2$?
\end{named}

\begin{rem}\label{rem-quesI}
Note that Example \ref{ex} of the Thompson group shows that the condition that $G$ has finite cohomological dimension is necessary. On the other hand, by \cite[Corollary 2]{St95},  the question is the same as asking whether $G$ is a Poincar\'e duality group.
\end{rem}

We have an affirmative answer to this question in the case when $G$ is of type F.
\begin{named}{Theorem E}\label{thm-existence}
Let $G$ be a group of type F, with $H^{i}(G, \mathbb ZG)$ finitely generated for all $i$. Then
\begin{enumerate}
\item there exists a closed manifold $M$ with $\pi_{1}(M) =G$ and $\pi_{i}(M)$  finitely generated for all $i \ge 2$;
\item $G$ is a  Poincar\'e duality group.
\end{enumerate}
\end{named}

The second part of \ref{thm-existence} recovers  partially a theorem of Farrell \cite[Theorem 3]{Fa75}, see Remark \ref{rmk: farrell}.

\begin{rem}
It is an open question whether a finitely presented Poincar\'e duality group is always the fundamental group of a closed aspherical manifold, see \cite{Da00} and \cite{Lu08} for more information.
\end{rem}

\textbf{Acknowledgements.} Su would like to thank the Max-Planck Institute for Mathematics at Bonn for a research visit in August 2018. Su was partially supported by NSFC 11571343.  Wu was partially supported by Prof.~Wolfgang L\"uck's ERC Advanced Grant ``KL2MG-interactions'' (no.  662400) and the DFG Grant under Germany's Excellence Strategy - GZ 2047/1, Projekt-ID 390685813. The authors would like to thank Ross Geoghegan for some helpful communications.

\section{Basic definitions and results}
In this section we collect some basic definitions and results that we may need later. For more details see \cite{Br82} and \cite{Kl99}.

\begin{defn}\label{defn-fd}
A CW-complex $X$ is called \textbf{finitely dominated} if there is a finite CW-complex $K$ and two maps $i:X \rightarrow K$, $r:K\rightarrow X$ such that $r\circ i$ is homotopic to $id_{X}$.
\end{defn}

\begin{defn}\label{defn-pdc}
A finitely dominated  CW-complex $X$ is called a \textbf{Poincar\'e duality space} of dimension $n$ if there is a $\BZ \pi_1(X)$-module structure on $\BZ$ and $e\in H_n(X,\BZ)$ such that the cap-product $e \cap -$
$$ H^i(X,A)\rightarrow H_{n-i}(X,\BZ \otimes A)$$
induces isomorphisms for all $i$ and all $\BZ \pi_1(X)$-modules $A$.
\end{defn}

\begin{defn}\label{defn-type}
A group $G$ is called of \textbf{type F} if it has a finite CW-complex model for the Eilenberg--MacLane space $K(G,1)$. A group $G$ is called of \textbf{type $F_{\infty}$} if it has a  CW-complex model for the Eilenberg--MacLane space $K(G,1)$ with finitely many cells in each dimension. A group $G$ is called of \textbf{type FP} if the trivial $\mathbb Z G$-module $\mathbb Z$ has a projective resolution of finite type over $\mathbb Z G$.
\end{defn}

\begin{defn}\label{defn-pdg}
A group $G$ is called a \textbf{Poincar\'e duality group} of dimension $n$ if $K(G,1)$ is a Poincar\'e duality space of dimension $n$. A group  $G$ of type FP is called a \textbf{duality group} if there is an integer $n$ such that $H^{i}(G, \mathbb Z G)=0$ for all $i \ne n$ and $H^{n}(G, \mathbb Z G)$ is a torsion-free abelian group.
\end{defn}

Note that a Poincar\'e duality group is a duality group with $H^{n}(G, \mathbb Z G) \cong \mathbb Z$  \cite[Section VIII.10]{Br82}.

\begin{lem}\label{fd-fg-htp}
A simply connected finite dimensional CW-complex  $X$ is finitely dominated if and only if $X$ is homotopy equivalent to a finite CW-complex, if and only if $\pi_i(X)$ is finitely generated for all $i$.
\end{lem}
\Proof The first ``if and only if" follows from Wall's finiteness theorem as $X$ is simply connected. Now if $X$ is homotopy equivalent to a finite CW-complex, then by Serre's mod $\mathscr C$ theory (\cite{Se53} or \cite[Chapter 9 Section 6]{Sp66}), we have $\pi_i(X)$ is finitely generated for all $i$. For the other direction, if $\pi_i(X)$ is finitely generated for all $i$, again by Serre's Theorem, we have $H_i(X)$ is finitely generated for all $i$. In this case $X$ has a minimal cell structure consisting  of finitely many cells in each dimension \cite[Proposition 4C.1]{Ha02}. Since $X$ is finite dimensional, it is homotopy equivalent to a finite CW-complex.
\qed

Combining Lemma \ref{fd-fg-htp} with  \cite[Theorem 4.1]{St96}, we have the following.

\begin{prop}\label{prop:fgh-finf}
Let $X$ be a finite CW-complex and $\tilde{X}$ be its universal cover. Suppose $\pi_i(\tilde{X})$ is finitely generated for all $i$, then $\pi_1(X)$ is a group of type $F_{\infty}$.
\end{prop}
\Proof
By Lemma \ref{fd-fg-htp} we have the universal over of $X$ is homotopy equivalent to a finite CW-complex. By \cite[Theorem 4.1]{St96}, $\pi_1(X)$ is of type $FP_{\infty}$. By \cite[Proposition VIII.4.3]{Br82}, we have $\pi_1(X)$ is of type $FL_{\infty}$. But since $X$ is a finite CW-complex, $\pi_1(X)$ is finitely presented. Hence by \cite[Theorem VIII.7.2]{Br82}, $\pi_1(X)$ is of type $F_{\infty}$.
\qed

\section{A technical theorem} \label{section:tech-thm}

In this section, we prove a technical theorem,   which provides the algebraic ingredients for the proof of our main results.  We first need  a sequence of lemmas for calculating the $\Ext$ functor.

\begin{lem} \label{free-ext-cal}
Let $G$ be a group such that $H^i(G,\BZ G)$ is finitely generated for a given $i$, $A $ be a  $\BZ G$-module whose underlying abelian group is finitely generated free. Then $\Ext^i_G(A,\BZ G)$ is also finitely generated.
\end{lem}

\Proof
Note first that when $A = \BZ$  and $G$ acts  trivially, $\Ext^i_G(A,\BZ G) \cong H^i(G,\BZ G)$ which is finitely generated by assumption.

In general we have $\Ext^{i}_G (A,\BZ G) \cong H^{i}(G,\Hom(A,\BZ G))$ (c.~f.~ \cite[III.2.2, p.61]{Br82}), where $G$ acts diagonally on $\Hom(A,\mathbb Z G)$. There are the following $\BZ G$-module isomorphisms
 $$\Hom (A,\BZ G) \cong \Hom(A,\BZ) \otimes \BZ G \cong \Hom(A, \BZ)_0 \otimes \BZ G \cong (\BZ G)^{r}$$
where $\Hom(A,\BZ)_0$ is the trivial $\BZ G$-module and the second ismorphism is by \cite[III.5.7, p.69]{Br82}, and $r$ is the rank of $A$ as a free abelian group. Therefore
$$\Ext^i_G(A,\BZ G) \cong H^{i}(G, (\BZ G)^{r}) \cong H^{i}(G, \mathbb Z G)^{r}$$
and the lemma now follows.
\qed

\begin{lem}\label{fin-ext-cal-gel}
Let $G$ be a group and $i$ be an integer such that $H^i(G,\BZ G)$ and $H^{i-1}(G,\BZ G)$ are both finitely generated. Let $A $ be a $\BZ G$-module whose underlying abelian group is finite. Then $\Ext^i_G(A,\BZ G)$ is also finitely generated.
\end{lem}

\Proof We first show the lemma when $G$ acts on $A$ trivially. For that, we only need to prove the case when $A$ is a finite cyclic group. Suppose $A \cong \BZ/k$, we have a short exact sequence (of trivial $G$-modules)
$$ 0 \rightarrow \BZ \stackrel{ k}{\rightarrow} \BZ \rightarrow \BZ/k \rightarrow 0$$
Apply the functor $\Hom_{G}(-,\BZ G)$, we get a long exact sequence
$$ \cdots \rightarrow \Ext^i_G(\BZ/k,\BZ G)\rightarrow \Ext^i_G(\BZ,\BZ G) \rightarrow  \Ext^i_G(\BZ,\BZ G) \rightarrow  \Ext^{i+1}_G(\BZ/k,\BZ G) \cdots$$
By the assumption $\Ext^i_G(\BZ,\BZ G)$ is finitely generated for  $i$ and $i-1$. Thus $\Ext^i_G(\BZ/k,\BZ G)$ is also finitely generated.

Now we deal with the general case. Since $A$ is a finite group, its automorphism group is also finite. Thus, we can choose a finite index subgroup $H$ of $G$ such that $H$ acts trivially on $A$.  Note that since $H$ is a finite index subgroup of $G$, $\Hom_H(\mathbb Z G, \mathbb Z H) \cong \mathbb Z G \otimes_{\BZ H} \mathbb Z H \cong \mathbb Z G$ (c.~f.~\cite[III.5.9]{Br82}). By Eckmann-Shapiro Lemma (c.~f.\cite[Corollary 2.8.4]{Ben}), we have
$$ \Ext^i_H(A,\BZ H) \cong \Ext_G^i(A,\Hom_H(\BZ G,\BZ H))   \cong  \Ext_G^i(A,\mathbb Z G)$$
therefore $\Ext_G^i(A, \BZ G)$ is finitely generated.
\qed

Note that the arguments in the proof also show the following lemma, since by definition $H^i(G,\BZ G) = \Ext_G^i(\BZ,\BZ G)$ where $\BZ$ is a constant $\BZ G$-module.

\begin{lem}\label{fi-hmg-fc}
Let $H$ be a finite index subgroup in $G$, then for any given $i$,  $H^i(G,\BZ G)$ is finitely generated  if and only if $H^i(H,\BZ H)$ is finitely generated.
\end{lem}

\begin{prop} \label{fg-ext-cal}
Let $G$ be a group and $i$ be an integer such that $H^i(G,\BZ G)$ and $H^{i-1}(G,\BZ G)$ are both finitely generated, $A$ be a $\BZ G$-module whose underlying abelian group is finitely generated. Then $\Ext^i_G(A,\BZ G)$ is also a finitely generated abelian group.
\end{prop}

\Proof
Since the torsion subgroup $T$ of $A$ is a $\BZ G$ submodule of $A$, we have the follow short exact sequence
$$ 0 \rightarrow T \rightarrow A \rightarrow A/T \rightarrow 0$$
Now as abelian groups, $T$ is finite, $A/T$ is finitely generated free. Apply the functor $ \Hom_G(-,\BZ G)$, we get a long exact sequence, and the proposition follows now from Lemma \ref{free-ext-cal} and Lemma \ref{fin-ext-cal-gel}.

\qed

Note that the proof of Proposition \ref{fg-ext-cal} also shows the following which will be useful later to prove part (a) of \ref{mthm-gel-mfd}.

\begin{prop}\label{ext-vanish}
Let $G$ be a group such that $H^{i}(G,\BZ G) = 0$ for any $i\leq n$ and $A$ is a finitely generated abelian group with a $\BZ G$-module structure.  Then $\Ext^i_G(A,\BZ G) = 0$ for any $i\leq n$.
\end{prop}

We are now ready to show the following.

\begin{thm} \label{thm-htp-fin}
Let $M$ be a closed manifold of dimension $n$ with fundamental group $G$, and $H^i(G,\BZ G)$ is finitely generated for all $i\leq [n/2]$. If for $2 \le i\leq [n/2]$, $\pi_i(M)$ is finitely generated, then $\pi_i(M)$ is finitely generated for all $i$.
\end{thm}

\begin{proof}  We may  assume $M$ is orientable by passing to an index two cover. In doing so the finite generation of $H^i(G,\BZ G)$ is not affected by Lemma \ref{fi-hmg-fc}. By Serre's theorem the universal cover $\widetilde{M}$ has finitely generated $\pi_i(M) \cong \pi_i(\widetilde{M})$  for all $i \ge 2$ if and only if $H_i(\widetilde{M},\BZ)$ is finitely generated for all $i \ge 2$ \cite[Chapter 9 Section 6 Theorem 21]{Sp66}. Let $\Lambda=\mathbb Z \pi_{1}(M)$ be the group ring, fix a CW-structure on $M$ (see for example \cite[Corollary A.12]{Ha02}), then $H_{i}( \widetilde M;\mathbb Z)$ is isomorphic to the cellular homology $H_{i}(M, \Lambda)$. Since $M$ is  closed and orientable, by Poincar\'e duality, we have $H_i(M, \Lambda) \cong  H^{n-i}(M, \Lambda )$. So we only need to show for $i \leq  [n/2]$, $H^{i}(M,\Lambda)$ is finitely generated.

We have a Universal Coefficient Spectral Sequence   (see \cite[Chapter I, Theorem 5.5.1]{Go58} or \cite[Theorem (2.3)]{Le77}), which converges to $H^{i}(M,\Lambda)$,  with $E_{2}$-terms
$$E_2^{p,q} = \Ext_\Lambda^{q}(H_p(\widetilde{M}),\Lambda) $$
Since we know already that $\pi_p(M)$  is finitely generated for $2 \le p\leq [n/2]$, by Serre's theorem $H_p(\widetilde{M})$ as an abelian group is finitely generated for $p\leq [n/2]$. Now since $H^i(G,\BZ G)$ is finitely generated for any $i\leq [n/2]$, by Proposition \ref{fg-ext-cal},  each term in the $E_{2}$-page of the spectral sequence is also finitely generated as long as $p\leq [n/2]$ and $q\leq [n/2]$. Thus for $i\leq [n/2]$, the limit group $H^{i}(M,\Lambda)$ is finitely generated.
\end{proof}

\begin{rem}
Note that Poincar\'e duality groups satisfy the condition  that $H^i(G,\BZ G)$ is finitely generated for all $i$; duality groups also satisfy the condition in our theorem if the dimension of the group is $> [n/2]$ (see Definition \ref{defn-pdg}).
\end{rem}

\begin{rem}
When $G$ does not satisfy the condition that $H^i(G,\BZ G)$ is finitely generated for any $i\leq [n/2]$, then Theorem \ref{thm-htp-fin} is not true in general. For example, take $M$ to be the connected sum of two copies of $S^1\times S^d$ for some $d\geq 3$.  Then $\pi_i(M) = 0$ for $2\leq i\leq d-1 $ but $\pi_d(M)$ is a free abelian group of infinite rank.
\end{rem}

\section{The manifold case}
In this section we prove our main theorems in the manifolds case. There are two main ingredients in the proof. The algebraic one we have already presented in Section \ref{section:tech-thm}. The geometric one is the following theorem on fibration of Poincar\'e duality spaces,  first announced by Quinn \cite[Remark 1.6]{Qu72}, for a proof see Gottlieb \cite{Go79} or Klein \cite[Corollary F]{Kl01}.

\begin{thm}\label{pd-comlex-fib}
Let $F\rightarrow E\rightarrow B$ be a fibration such that $F,E,B$ are finitely dominated CW-complexes. Then $E$ is a Poincar\'e duality space if and only if $F$ and $B$ are Poincar\'e duality spaces. When $E$ is a Poincar\'e duality space of dimension $n$, the sum of the duality dimensions of $F$ and $B$ is also $n$.
\end{thm}

\begin{cor}\label{weak-thma}
Let $M^n$ be a closed manifold such  that
$\pi_1(M)$ is the fundamental group of a finite aspherical CW-complex $B$.
\begin{enumerate}
\item  If $B$ is not a Poincar\' e duality space, then $\pi_{i}(M)$ is not finitely generated for some $i \ge 2$;
\item If $B$ is a Poincar\'e duality space of dimension $d$ with $ d \geq  n - 1$, then either $\pi_i(M)$ is not finitely generated for some $i\geq 2$, or $M$ is homotopy equivalent to $B$.
\end{enumerate}
\end{cor}

\proof

Since $\pi_1(M) \cong \pi_1(B)$, we have a map $f: M\rightarrow B$ which induces isomorphism on $\pi_1$. Let $F_f$ be its homotopy fiber, then $F_f$ is homotopy equivalent to the universal cover $\widetilde M$ of $M$. Assume now $\pi_i(M)$ is finitely generated for all $i \ge 2$, by Lemma \ref{fd-fg-htp}, $\widetilde{M}$ is homotopy equivalent to a finite CW-complex. Now by Theorem \ref{pd-comlex-fib}, we have $F_f$ and $B$ are Poincar\'e duality spaces and the duality dimension of $F_f$ is $n-d$. Note that by assumption $n-d\leq 1$. If $n-d=0$, then $F_f$ is homotopy equivalent to a point since it is simply connected. Therefore, $M$ and $B$ are homotopy equivalent. When $n-d =1$, then $H_1(F_f,\BZ) \cong \BZ$ since it is a Poincar\'e duality space, which is contradicting to the fact that $F_f$ is simply connected.
\qed

Note that  now \ref{mthm-pdg-mfd} follows from  Corollary \ref{weak-thma} and Theorem \ref{thm-htp-fin}, as duality groups of dimension $d \ge n-1$ satisfy the condition that $H^i(G,\BZ G)$ is finitely generated for $i \le [n/2]$. The only exceptional case is when $n=2$. But in this case the theorem follows directly from the classification of surfaces.

We now proceed to prove  \ref{mthm-gel-mfd}.

\begin{proof}[Proof of  part (a) of \ref{mthm-gel-mfd} ]
Suppose $\pi_i(M)$ is finitely generated for any $i$, by Serre's theorem, $H_i(\widetilde{M})$ is finitely generated for any $i$.
We have a Universal Coefficient Spectral Sequence, which converges to $H^{i}(M,\BZ G)$,  with $E_{2}$-terms
$$E_2^{p,q} = \Ext_{\BZ G}^{q}(H_p(\widetilde{M}),\BZ G) $$
Since $H^{i}(G,\BZ G) = 0$ for any $i\leq d$, we have $\Ext_{\BZ G}^{q}(H_p(\widetilde{M}),\BZ G) = 0$ for any $q\leq d$ by Proposition \ref{ext-vanish}. Hence $H^i(M,\BZ G) = 0$ for any $i\leq d$.

If $n =d$, then $ H_i(\widetilde{M},\BZ) \cong  H_i(M,\BZ G) \cong H^{n-i}(G, \BZ G) = 0$ for any $i$. But this is a contradiction since $ H_0(\widetilde{M},\BZ) \cong \BZ$.

Now we assume $n=d+1$. Then  $H_i(\widetilde{M},\BZ) \cong H_i(M,\BZ G) \cong  H^{d-i}(M,\BZ G) = 0$ for any $i \geq 1$. This implies $M$ is aspherical, in particular, $G$ is a Poincar\'e dualtiy group.  The same argument works for $n = d+2$ as we know already  $\widetilde{M}$ is simply connected.
\end{proof}

\begin{rem}
Note that for a finite group $G$, we have $H^0(G,\BZ G) \cong \BZ$ \cite[Proposition 13.2.11]{Ge08} and $H^i(G,\BZ G) =0$ for any $i>0$ \cite[Proposition 13.3.1]{Ge08}. On the other hand, when $G$ is not finite, $H^0(G,\BZ G) = 0$ \cite[Proposition 13.2.11]{Ge08}.
\end{rem}

\begin{rem}
Part (a) of \ref{mthm-pdg-mfd} now also follows from part (a)  of \ref{mthm-gel-mfd}  and Theorem \ref{thm-htp-fin} which are independent of  Theorem \ref{pd-comlex-fib}.
\end{rem}

Part (b) of \ref{mthm-gel-mfd}  is a special case of the following theorem using Lemma \ref{fd-fg-htp}.

\begin{thm} \label{fd-g-homol}
Let $X$ be an $n$-dimensional finite Poincar\'e complex. Let $N$ be a normal subgroup of $\pi_1(X)$ with quotient $G$, and $X_N$ be the corresponding cover. If $X_N$ is finitely dominated, then $H^i(G,\BZ G)$ is a finitely generated abelian group for all $i$.
\end{thm}

\begin{proof}
If $G$ is a finite group, the theorem automatically holds. So we assume from now on $G$ is an infinite group, in particular $H_n(X_N,\BZ) = 0$.

Apply the Leray-Serre spectral sequence to the fibration $ X_N \to X \to BG$, we have $E_2^{p,q}=H^p(G,H^q(X_N;\BZ G))$ and the spectral sequence coverges to the graded groups of a filtration of $H^i(X;\BZ G)$.  By Poincar\'e  duality and \cite[Corollary 13.2.3]{Ge08} $H^i(X;\BZ G) \cong H_{n-i}(X, \BZ G) \cong H_{n-i}(X_N)$ which is a finitely generated abelian group since $X_N$ is finitely dominated. Note also that  $H^p(G,H^0(X_N;\BZ G)) \cong H^p(G,\BZ G)$.

For $i=0$, we have $0=H_n(X_N,\BZ)  = E_{\infty}^{0,0}=E_{2}^{0,0}=H^0(G,H^0(X_N;\BZ G)) \cong H^0(G, \BZ G)$.

For $i=1$, we have $E_{\infty}^{1,0} = E_{2}^{1,0} = H^1(G,H^0(X_N;\BZ G)) \cong H^1(G, \BZ G)$, $E_{\infty}^{1,0}$ is a subgroup of $ H_{n-1}(X_N)$ hence finitely generated.

We proceed by induction. Assume that $H^i(G, \BZ G)$ is finitely generated for $i < k$. Now $E_{\infty}^{k,0}$ is a subgroup of $H_{n-k}(X_N)$, hence finitely generated. Note that $E_{\infty}^{k,0}$ is a quotient of $E_2^{k,0}=H^k(G, H^0(X_N; \BZ G))=H^k(G, \BZ G)$. The differentials ending at the position $(k,0)$ come from the line $p+q=k-1$. Hence we only need to show that all the $E_2$-terms $E_2^{p,q}=H^p(G, H^q(X_N;\BZ G))$ are finitely generated for $p<k-1$.

By the universal coefficient theorem there is a short exact sequence
$$0 \to \Ext^1(H_{q-1}(X_N), \BZ G) \to H^q(X_N, \BZ G) \to \Hom(H_q(X_N), \BZ G) \to 0$$
which is an exact sequence of $G$-modules by naturality. This induces a long exact sequence
$$\cdots \to H^p(G, \Ext^1(H_{q-1}(X_N), \BZ G)) \to H^p(G, H^q(X_N,\BZ G)) \to H^p(G, \Hom(H_q(X_N), \BZ G))  \cdots$$
We have $\Hom(H_q(X_N), \BZ G) \cong \Hom(H_q(X_N), \BZ) \otimes \BZ G \cong (\BZ G)^r$ for some integer $r$. Therefore by the inductive assumption $H^p(G, \Hom(H_q(X_N), \BZ G))$ is finitely generated for $p < k$ and any $q$. We only need to show $H^p(G, \Ext^1(H_{q-1}(X_N), \BZ G))$ is finitely generated for $p <k-1$.

Notice that $H_{q-1}(X_N)$ is a finitely generated abelian group, we have $\Ext^1(H_{q-1}(X_N), \BZ G) =  \Ext^1(A, \BZ G)  $, where $A$ is the torsion part of $H_{q-1}(X_N)$. Here $A$ as an abelian group is finite. Let $A_0$ be the  underlying abelian group of $A$. Then as $G$-modules $\Ext^1(A, \BZ G) \cong \Ext^1 (A, \BZ ) \otimes \BZ G \cong A \otimes \BZ G \cong A_0 \otimes \BZ G$. Let $0 \to \BZ^s \to \BZ^s \to A_0 \to 0$ be a free resolution of $A_0$ over $\BZ$, then from the short exact sequence of $G$-modules
$$0 \to (\BZ G)^s \to (\BZ G)^s \to A_0 \otimes \BZ G \to 0$$
and the assumption that $H^p(G, \BZ G)$ is finitely generated for $p <k$, it's easy to see that $H^p(G, A_0 \otimes \BZ G)$ is finitely generated for $p < k-1$. Therefore $H^p(G, \Ext^1(H_{q-1}(X_N), \BZ G))$ is finitely generated for $p <k-1$ by Proposition \ref{fg-ext-cal}.  This finishes the proof.
\end{proof}

\section{The finite CW-complex case}
In this section, we prove our main results for finite CW-complexes. We also discuss the problem of constructing manifolds with finitely dominated universal cover.

We first generalize \cite[Theorem 1.2]{Da09} from polycylic groups to any virtually Poincar\'e duality groups.

\textbf{Proof of \ref{mthm-pdg-cw}.}
With the help of Theorem \ref{pd-comlex-fib} and   \ref{thm-htp-fin}, the proof now follows similarly to the arguments in \cite[p.1797-1798]{Da09}. We will assume $\pi_1(M)$ is a Poincar\'e duality group. The general case follows easily from this. In fact, we will assume that $\pi_1(M)$ is an orientable Poincar\'e duality group (i.e. the $\BZ G$-module structure on $\BZ$ is trivial in Definition \ref{defn-pdg}) as we can always pass to an orientable index two subgroup.

Suppose $\pi_{i}(X)$ is finitely generated for all $i \ge 2$.  By simplicial approximation we may assume that $X$ is a finite simplicial complex. Embed $X$ in an Euclidean space $\BR^{2n+r+1}$ for some $r\geq 0$ which will be fixed later in the proof. Let $W$ be a regular neighbourhood of $X$ and denote by $M^{2n+r}$ the boundary of $W$.  Then it is a standard fact that $X$ is a deformation retract of $W$, and the inclusion map $M \to W$ is $(n+r)$-connected.  Therefore $\pi_i(M)$ is finitely generated for $i\leq n+r-1$. If $r\geq 2$, then $n+r-1 \geq [(2n+r)/2]$,  by  Theorem \ref{thm-htp-fin} $\pi_i(M)$ is finitely generated for all $i$. Therefore $\widetilde{M}$ is homotopy equivalent to a finite CW-complex.  Apply Theorem \ref{pd-comlex-fib} to the fibration sequence
$\widetilde{M} \longrightarrow M\longrightarrow B\pi_1(X)$,
where $B\pi_1(X)$ is a model for $K(\pi_1(X),1)$, we see that  $\widetilde{M}$ is homotopy equivalent to a Poincar\'e duality complex of dimension $2n+r-d$, in particular $H_{2n+r-d}(\widetilde{M}) \cong \BZ$.

Case (a): when $d>n$, we have $2n+r-d \le n+r-1$, hence
$ H_{2n+r-d}(\widetilde{M}) \cong H_{2n+r-d}(\widetilde{W}) \cong H_{2n+r-d}(\widetilde{X})$ by the Whitehead theorem. On the other hand, if we  choose $r$ to be bigger than $d-n$, we have $2n+r-d >n = \dim{X}$, so $ H_{2n+r-d}(\widetilde{X})$ vanishes. This leads to a contradiction.

Case (b):  when $d=n$ or $n-1$,  the Poincar\'e duality dimension of $\widetilde{M}$ is $n+r+1$ or $n+r$. Since the map $M \to W$ is $(n+r)$-connected and $X$ is an $n$-dimensional complex, $H_i(\widetilde{M},\BZ)$ vanishes for $i =n+1,n+2,\cdots, n+r-1$. By the universal coefficient theorem and  Poincar\'e duality it is easy to see that $H_i(\widetilde{M},\BZ)$ also vanishes for $i =n-d+1,\cdots, n-d+r-2$. Now $n-d+1=1$ or $2$,  and $\widetilde{M}$ is simply connected, hence for $r$ sufficiently large $H_{i}(\widetilde{M})=0$ for all $1 \le i\leq n$. But  $H_{i}(\widetilde{X}) \cong H_{i}(\widetilde M)$ for $i \le n$, this implies that $\widetilde X$ is contractible. Therefore $X$  is a $K(\pi_1(X),1)$ space.
\qed

The first part of \ref{thm-existence} is a special case of the following theorem.

\begin{thm}\label{thm-relz}
Let $X$ be a $n$-dimensional finite CW-complex such that $\pi_1(X)=G$ and $\pi_i(X)$ is  finitely generated for all $i$. If $H^i(G,\BZ G)$ is finitely generated for all $i\leq n$, then we can find a closed manifold $M$ of dimension $2n+1$ such that $\pi_1(M)=G$ and $\pi_i(M)$ is finitely generated for any $i$.
\end{thm}
\Proof We can assume $n \geq 2$. Embed $X$ into an Euclidean space $\BR^{2n+2}$. Let $W$ be a regular neighbourhood of $X$ and denote by $M^{2n+1}$ the boundary of $W$. Similar to the proof of \ref{mthm-pdg-cw}, $W$ is a deformation retract of $X$ and the inclusion map $M \to W$ is $(n+1)$-connected, therefore $\pi_i(M)$ is finitely generated for any $2 \le i\leq n$. Note that $n \geq [\frac{2n+1}{2}]$, by Theorem \ref{thm-htp-fin}, we have $\pi_i(M)$ is finitely generated for all $i \ge 2$.
\qed

\begin{rem}
Note that Theorem \ref{thm-relz} implies that, to answer \ref{ques}, it suffices to find a finite CW-complex $X$ with fundamental group $G$ such that $\pi_i(X)$ finitely generated for all $i$.
\end{rem}

\begin{cor}
Let $X$ be a finite CW-complex with fundamental group the Thompson group $F$, then for some $i\geq 2$, $\pi_i(X)$ is not finitely generated.
\end{cor}
\Proof Suppose $\pi_i(X)$ is finitely generated for all $i$. Since $H^i(F,\BZ F)=0$ for any $i$ \cite[Theorem 7.2]{BG84},  by Theorem \ref{thm-relz}, we have a manifold $M$ with fundamental group $F$ such that $\pi_i(M)$ is finitely generated for all $i$. Now by \ref{mthm-gel-mfd} (a), we have $M$ is aspherical and  $F$ is a Poincar\'e duality group which is a contradiction.
\qed

\begin{proof}[Proof of  \ref{mthm-gel-cw} ]
Suppose $\pi_i(X)$ is finitely generated for any $i \ge 2$, Theorem \ref{thm-relz} implies there is a closed manifold $M^{2n+1}$ with fundamental group $G$ such that $\pi_i(M)$ is finitely generated for all $i \ge 2$. But then part (b) of \ref{mthm-gel-mfd} says $H^i(G,\BZ G)$ must be finitely generated for all $i$.
\end{proof}

\begin{proof}[Proof of  \ref{thm-existence} (b)]
By part (a), we can find a  closed manifold $M$ such that $\pi_i(M)$ is finitely generated for all $i$. Hence the universal cover $\widetilde{M}$ of $M$ is homotopy equivalent to a finite CW-complex. Now consider the following fibration sequence
$$\widetilde{M} \longrightarrow M\longrightarrow B\pi_1(M)$$
Where $B\pi_1(M)$ is a model for $K(\pi_1(M),1)$. By Theorem \ref{pd-comlex-fib}, we have $B\pi_1(M)$ is a Poincar\'e duality space. Hence $G$ is a Poincar\'e duality group.
\end{proof}

\begin{rem}\label{rmk: farrell}
This recovers partially a result of Farrell \cite[Theorem 3]{Fa75}. Recall Farrell's theorem says the following: let $G$ be a group of type F and $n$ be the smallest integer such that $H^n(G,\BZ G) \neq 0$, if $H^n(G, \BZ G)$ is a finitely generated abelian group, then $G$ is Poincar\'e duality group.
\end{rem}

\bibliographystyle{amsplain}

\end{document}